\documentclass[12pt]{amsart}
\usepackage{amsmath,amsthm,amssymb,amscd,euscript}
\usepackage[all]{xy}
\usepackage{enumerate}
\usepackage[margin=2.6cm]{geometry}
\usepackage{hyperref}
\usepackage{graphicx}
\hypersetup{pdftex,colorlinks=true,allcolors=blue}

\newtheorem{thm}{Theorem}[section]

\theoremstyle{plain}
\newtheorem{lemma}{Lemma}[section]
\newtheorem{prop}[lemma]{Proposition}
\newtheorem{coro}[lemma]{Corollary}
\theoremstyle{definition}
\newtheorem{defn}[lemma]{Definition}
\newtheorem{remark}[lemma]{Remark}

\newcommand{\jkh}[1]{\left\langle #1 \right\rangle}
\newcommand{\ykh}[1]{\left( #1 \right)}
\newcommand{\hkh}[1]{\left\{ #1 \right\}}

\newcommand{\sqz}[1]{\left \lceil  #1 \right \rceil}
\newcommand{\xqz}[1]{\left \lfloor #1\right \rfloor}

\newcommand{\DB}{\operatorname{\mathcal{D}}^b}

\newcommand{\GL}{\operatorname{GL}}
\newcommand{\Hom}{\operatorname{Hom}}

\newcommand{\Stab}{\operatorname{Stab}}

\newcommand{\CA}{\mathcal{A}}
\newcommand{\CE}{\mathcal{E}}

\newcommand{\CK}{\mathcal{K}}
\newcommand{\CL}{\mathcal{L}}
\newcommand{\CO}{\mathcal{O}}
\newcommand{\CP}{\mathcal{P}}
\newcommand{\CR}{\mathcal{R}}
\newcommand{\CT}{\mathcal{T}}

\newcommand{\BC}{\mathbb{C}}
\newcommand{\BH}{\mathbb{H}}
\newcommand{\BP}{\mathbb{P}}
\newcommand{\BZ}{\mathbb{Z}}
\newcommand{\BR}{\mathbb{R}}

\begin{document}

\title{Spaces of stability conditions via exceptional collection: length 4 case}
\author{Yun-Feng Wang}

\address{School of Mathematical Sciences\\
University of Science and Technology of China, Hefei, Anhui, 230026, People’s Republic of China\\
}
\email{yfwang89@mail.ustc.edu.cn}

\date{}

\begin{abstract}
In this paper, we study the space of stability conditions on triangulated categories generated by an exceptional collection. We give an exact description of subspace of  stability conditions generated by length 4 complete exceptional collection. 
\end{abstract}

\maketitle

\section{Introduction}
The space of \emph{stability conditions} on a triangulated category $\CT$ was introduced by Bridgeland \cite{Bridgeland2007}, with inspiration from work by Douglas \cite{douglas2002dirichlet}  in string theory. Roughly speaking, a stability condition $\sigma$ on a triangulated category $\CT$ consists of data $(Z,\CP)$ where $Z$ is called the \emph{central charge}, and $\CP$ is a slicing on $\CT$. There are many examples and conjectures about the constructions of the stability conditions, we refer to surveys \cite{Bridgeland2009,macri2021,macri2017} for more information and applications.

Based on \cite{bondal1989representation,bondal1989representable,bernstein1982faisceaux}, Macr\`i \cite{macri2007stability,macri2004} gave a procedure generating the stability conditions on triangulated categories generated by an exceptional collection which was developed further by Collins and Polishchuk \cite{collins2010}. To a fixed complete exceptional collection $\CE=\hkh{E_0,\cdots,E_n}$ on a triangulated category $\CT$ , Macr\`i \cite{macri2007stability,macri2004}  obtained an subspace  of stability conditions, denoted by $\Theta_{\CE}$, generated by $\CE$. 

Motivated by the work of Macr\`i \cite{macri2007stability,macri2004}, Dimitrov and Katzarkov \cite{DimitrovGeorge2016,DimitrovGeorge2016imrn} introduced the notation $\sigma$-\emph{exceptional collection} and reformulated Macr\`i's construction of $\Theta_{\CE}$. For length 3 complete exceptional collection $\CE=\hkh{E_0,E_1,E_2}$, Dimitrov and Katzarkov \cite{DimitrovGeorge2016} gave an exact description of $\Theta_{\CE}$. For longer length cases, it is difficult to give an exact description of $\Theta_{\CE}$ following the approach in \cite{DimitrovGeorge2016}.  As explained in \cite[Remark 2.5]{DimitrovGeorge2016}, the main difficulty is that the similarly equality in \cite[Lemma 2.4]{DimitrovGeorge2016}  does not hold for $n\geq 3$. 

In this paper we give a description of $\Theta_{\CE}$ for length 4 exceptional collection $\CE$ as follows

\begin{thm}\label{thmmain}
Let $\CT$ be a finite linear triangulated category over $\BC$. Let $\CE=\hkh{E_0,E_1,E_2,E_3}$ be a complete exceptional collection, denoted by 
\[
k_{ij}=\min\hkh{k:\Hom^{k}(E_i,E_j)\neq 0}\in\BZ.
\]
\begin{equation*}
\begin{split}
\CK_{02}=\min\hkh{k_{02},k_{01}+k_{12}-1}, \quad \CK_{13}=\min\hkh{k_{13},k_{12}+k_{23}-1},\\
\CK_{03}=\min\hkh{k_{01}+k_{12}+k_{23}-2, k_{01}+k_{13}-1,k_{02}+k_{23}-1,k_{03}},
\end{split}
\end{equation*}
and
\[
\Theta_1=\hkh{
(y_0,y_1,y_2,y_3)\in\BR^4\Bigg|\begin{array}{cl} 
&y_0-y_1<k_{01},\quad y_1-y_2<k_{12}\\
&y_2-y_3<k_{23}, \quad y_1-y_3<\CK_{13}\\
&y_0-y_3<\CK_{03},\quad y_0-y_2<\CK_{02}
\end{array} 
}
\]
Then $\Theta_{\CE}$ is homeomorphic with the set 
\[
\BR_{>0}^4\times\ykh{ \Theta_1\setminus \Delta}.
\]
where $\Delta$ is the union of following five sets

\begin{eqnarray*}
&\Delta^1=\hkh{(x_0,x_1,x_2,x_3)\in\Theta_1\Bigg|
\begin{array}{cll} 
&\CR_{23}^1,\quad\CR_{13}^1\\
&\xqz{x_0-x_1}=k_{01}-1,\xqz{x_0-x_2}< \CK_{02}-1, \\
&\xqz{x_2-x_3}=k_{23}-1,x_1-x_2=\xqz{x_1-x_2} 
\end{array} 
}\\
&\Delta^2=\hkh{(x_0,x_1,x_2,x_3)\in\Theta_1\Bigg|
\begin{array}{cll} 
&\CR_{23}^1,\quad\CR_{13}^0,\quad \xqz{x_0-x_2}< \CK_{02}-1, \\
&\xqz{x_2-x_3}=k_{23}-1,x_1-x_2=\xqz{x_1-x_2} 
\end{array} 
}\\
&\Delta^3=\hkh{(x_0,x_1,x_2,x_3)\in\Theta_1\Bigg|
\begin{array}{clll} 
&\CR_{12}^1,\quad\CR_{13}^0, \quad \xqz{x_0-x_1}<k_{01}-1, \\
&\xqz{x_1-x_2}=k_{12}-1,x_2-x_3=\xqz{x_2-x_3} 
\end{array} 
}\\
&\Delta^4=\hkh{(x_0,x_1,x_2,x_3)\in\Theta_1\Bigg|
\begin{array}{clll} 
&\CR_{12}^1,\quad\CR_{13}^0, \quad \xqz{x_0-x_1}<k_{01}-1, \\
&\xqz{x_1-x_2}=k_{12}-1,\xqz{x_0-x_3} =\CK_{03}-1
\end{array} 
}\\
&\Delta^5=\hkh{(x_0,x_1,x_2,x_3)\in\Theta_1\Bigg|
\begin{array}{clll} 
&\CR_{12}^0,\quad\CR_{13}^1,\quad\xqz{x_1-x_3}=\CK_{13}-1,\\
&\xqz{x_0-x_1}<k_{01}-1,\xqz{x_0-x_2}= \CK_{02}-1
\end{array} 
}
\end{eqnarray*}
where $\CR_{ij}^{\alpha},\alpha\in\hkh{0,1}, 1\leq i,j\leq 3$ are the relations defined by $\eqref{relationR}$.
\end{thm}

{\textbf{Proof idea.}} Our method is different from \cite{DimitrovGeorge2016}. We refer to Section \ref{mainproof} for complete notations. Our proof based on two ideas.  First, we can see that $x_0=y_0$, then we need to $z_i\in (x_0-1,x_0+1)$, then the possible choices of $p_i$ are $\xqz{x_0-x_i}$ or $\sqz{x_0-x_i}$, then we reduce to check conditions $(z_0,z_1,z_2,z_3)\in S^3(-1,1)$ and $(0,p_1,p_2,p_3)\in A_0$, secondly, our main technical observation is   the properties  of floor function $\xqz{x}$ (See Lemma \ref{pair1}, Lemma \ref{triple1}, Lemma \ref{triple3}, Corollary \ref{triple2} and Corollary \ref{pijcor}) which can classify $\Theta_1$ into different cases, and check them case by case (See Lemma \ref{lemmacase1}-\ref{lemmacase4}). This classification can give all elements of $\Delta$ and avoid the more technical arguments of \cite[Lemma 2.4]{DimitrovGeorge2016}.


The paper is structured as follows. In the next section we review the basic notations and constructions. In Section \ref{mainproof} we give the complete  proof. In Section \ref{example} we give the examples and some remarks of author's original motivation.


\section{Review: basic definitions and constructions}

In this section we review the basic notations and facts concerning the construction of the stability conditions following Macr\`i  \cite{macri2007stability,macri2004}, Dimitrov--Katzarkov \cite{DimitrovGeorge2016,DimitrovGeorge2016imrn} and Bridgeland \cite{Bridgeland2007}.  For details the reader is referred to the original papers \cite{macri2007stability,macri2004,DimitrovGeorge2016,DimitrovGeorge2016imrn,Bridgeland2007,bondal1989representation}.

Let $\CT$ be a linear triangulated category and of finite type over $\BC$. We denoted by $\jkh{E}$ the extension closed subcategory of $\CT$ generated by subcategory $E$ of $\CT$ and $K(\CT)$ the Grothendieck group of $\CT$. Similarly, the Grothendieck group of an abelian category $\CA$ is denoted $K(\CA)$.

\subsection{Stability conditions}

\begin{defn}[\cite{Bridgeland2007}]
A \emph{stability condition} $\sigma=(Z,\CP)$ on a triangulated category $\CT$ consists of a group homomorphism $Z:K(\CT)\to \BC$, and full additive subcategories $\CP(\phi)\subset \CT$ for each $\phi\in\BR$, satisfying the following axioms:
\begin{enumerate}[(a)]
\item if $E\in\CP(\phi)$ then $Z(E)=m(E)\exp(i\pi\phi)$ for some $m(E)\in\BR_{>0}$,

\item for all $\phi\in\BR$, $\CP(\phi+1)=\CP(\phi)[1]$,

\item if $\phi_1>\phi_2$ and $A_j\in\CP(\phi_j)$ then $\Hom_{\CT}(A_1,A_2)=0$,
 
\item for each nonzero object $E\in\CT$ there are a finite sequence of real numbers 
\[
\phi_1 > \phi_2 > \cdots>\phi_n
\]
and a collection of triangles
\[
\xymatrix@C=.5em{
0_{\ } \ar@{=}[r] & E_0 \ar[rrrr] &&&& E_1 \ar[rrrr] \ar[dll] &&&& E_2
\ar[rr] \ar[dll] && \ldots \ar[rr] && E_{n-1}
\ar[rrrr] &&&& E_n\ar[dll] \ar@{=}[r] &  E_{\ } \\
&&& A_1 \ar@{-->}[ull] &&&& A_2 \ar@{-->}[ull] &&&&&&&& A_n \ar@{-->}[ull] 
}
\]
with $A_j\in \CP(\phi_j)$ for all $j$.
\end{enumerate}
\end{defn}

Nonzero objects in $\CP(\phi)$ are said to be \emph{semistable} of phase $\phi$. For interval $I\subseteq\BR$, we denote $\CP(I)$ the extension closed subcategory of $\CT$ generated by the subcategories $\CP(\phi),\phi\in I$. If there exists some $\varepsilon>0$ such that each $\CP((\phi-\varepsilon,\phi+\varepsilon))$ is of finite length, then we call the stability condition is \emph{locally finite}.  

By \cite[Theorem 1.2, Corollary 1.3]{Bridgeland2007}, the space of locally finite stability conditions, denoted by $\Stab(\CT)$, is a complex manifold. By \cite[Lemma 8.2]{Bridgeland2007},    $\Stab(\CT)$ carries a right action of the group $\widetilde{\GL}^+(2,\BR)$, the universal covering space of $\GL^+(2,\BR)$, and a left action by isometries of the group of exact autoequivalences of $\CT$.

For an abelian category $\CA$, Bridgeland \cite[Definition 2.1]{Bridgeland2007} defined the \emph{stability function} $Z: K(\CA)\to \BC$, and \cite[Definition 2.2]{Bridgeland2007} introduced  the \emph{Harder-Narasimhan property} for the stability function.  Then Bridgeland \cite[Proposition 5.3]{Bridgeland2007} showed that to give a  stability condition on a triangulated category $\CT$ is equivalent to giving a bounded $t$-structure on $\CT$ and a stability function on its heart with the  Harder-Narasimhan property.

\subsection{Exceptional objects}
The theory of exceptional collections  and helix theory developed in the Rudakov seminar \cite{rudakov1990helices}. In order to get a heart of a bounded $t$-structure on $\CT$ generated by the exceptional collection $\CE$, Macr\`i  \cite{macri2007stability,macri2004} introducted the notation \emph{Ext} exceptional collection. 

Now, we recall the basic definitions. Following \cite{bondal1989representation} we denote
\[
\Hom^{\bullet}(A,B)=\bigoplus_{k\in\BZ}\Hom^k(A,B)[-k],
\]
where $A,B\in\CT,\Hom^k(A,B)=\Hom(A,B[k])$. An object $E\in\CT$ is called \emph{exceptional} if it satisfies 
\begin{equation*}\Hom^i(E,E)=
\begin{cases}
\BC & \text{ if } i=0,\\
0 & \text{ otherwise. }
\end{cases}
\end{equation*}
An ordered collection of exceptional objects $\{E_0,\cdots,E_n\}$ is called \emph{exceptional} in $\CT$ if it satisfies
\[
\Hom^{\bullet}(E_i,E_j)=0,\text{ for } i>j.
\]
We call $\{E_0,\cdots,E_n\}$  the \emph{length}  $n+1$ exceptional collection, and we call an exceptional collection of two objects an \emph{exceptional pair}.

Let $(E,F)$ an exceptional pair. We define objects $\CL_E F$ and $\CR_F E$ (which we call \emph{left mutation} and  \emph{right mutation} respectively) by the distinguished triangles
\begin{eqnarray*}
\CL_E F\longrightarrow \Hom^{\bullet}(E,F)\otimes E\longrightarrow F,\\
E\longrightarrow \Hom^{\bullet}(E,F)^*\otimes F\longrightarrow \CR_F E,
\end{eqnarray*}
where $V[k]\otimes E$ denotes an object isomorphic to the direct sum of $\dim V$ copies of the object $E[k]$. A \emph{mutation} of an exceptional collection $\CE=\{E_0,\cdots,E_n\}$ is defined as a mutation of a pair of adjacent objects in this collection:
\[
\CR_i\CE=\{E_0,\cdots,E_{i-1},E_{i+1},\CR_{E_{i+1}}E_i, E_{i+2},\cdots,E_n\},
\]
\[
\CL_i\CE=\{E_0,\cdots,E_{i-1},\CL_{E_{i}}E_{i+1}, E_{i}, E_{i+2},\cdots,E_n\},
\]
for $i=0,\cdots,n-1$.

\begin{defn}[\cite{macri2007stability}]
Let $\CE=\{E_0,\cdots,E_n\}$ be an exceptional collection. We call $\CE$
\begin{itemize}
\item strong, if $\Hom^k(E_i,E_j)=0$ for all $i$ and $j$, with $k\neq 0$;
\item Ext, if $\Hom^{\leq 0}(E_i,E_j)=0$ for all $i\neq j$;
\item complete, if $\CE$ generates $\CT$ by shifts and extensions.
\end{itemize}
\end{defn}

\subsection{Basic construction}

Let $\CE=\hkh{E_0,\cdots,E_n}$ be a complete exceptional collection on $\CT$.   For $\textbf{p}=(p_0,\cdots,p_n)\in\BZ^{n+1}$, we denote by $\CE[\textbf{p}]:=\hkh{E_0[p_0],\cdots,E_n[p_n]}$ the shift of the exceptional collection $\CE$ and  $\CA_{\CE[\textbf{p}]}=\jkh{E_0[p_0],\cdots,E_n[p_n]}$ the extension closed subcategory generated by $\CE[\textbf{p}]$. We write 
\begin{equation}\label{A0def}
A_0=\hkh{\textbf{p}: \textbf{p}=(p_0,\cdots,p_n)\in\BZ^{n+1} \text{ such that } \CE[\textbf{p}] \text{ is Ext}}.
\end{equation}

From now, we assume $\textbf{p}\in A_0$. By \cite[Lemma 3.14]{macri2007stability} , $\CA_{\CE[\textbf{p}]}$ is the heart of a bounded $t$-structure on $\CT$. Fix $z_0,\cdots,z_n\in \BH:=\hkh{z\in\BC: z=|z|\exp(i\pi\phi),0<\phi\leq 1}$.  Define a stability function $Z_{\textbf{p}}: K_0(\CA_{\CE[\textbf{p}]})\to \BC$ by
\[
Z_{\textbf{p}}(E_i[p_i])=z_i, \text{ for all } i.
\]
By \cite[Lemma 3.16]{macri2007stability}, $\CA_{\CE[\textbf{p}]}$ is an abelian category of finite length and with finitely many simple objects, then by \cite[Proposition 2.4]{Bridgeland2007} any stability function $Z_{\textbf{p}}$ on $\CA_{\CE[\textbf{p}]}$ satisfies the Harder--Narasimhan property. By \cite[Proposition 5.3]{Bridgeland2007}, this extends to a unique stability condition on $\CT$ which is locally finite.

Following \cite[Defintion 3.13]{DimitrovGeorge2016imrn}, we denote by $\BH^{\CE[\textbf{p}]}$ the set of stability conditions on $\CT$ constructued above, denote $\Theta_{\CE[\textbf{p}]}^{\prime}=\BH^{\CE[\textbf{p}]}\cdot\widetilde{\GL}^+(2,\BR)$,  and 
\begin{equation}
\Theta_{\CE}=\bigcup_{\textbf{p}\in A_0} \Theta_{\CE[\textbf{p}]}^{\prime}\label{thetaall}.
\end{equation}

For the convenience, Dimitrov--Katzarkov \cite[Definition 2.3]{DimitrovGeorge2016} introduced the notation $S^n(-1,1)$, which defined by
\[
S^n(-1,+1)=\hkh{(y_0,\cdots,y_n)\in\BR^{n+1}: -1<y_i-y_j<1,i<j}\subset\BR^{n+1}.
\]

By the comment in Dimitrov--Katzarkov \cite[p. 833-834]{DimitrovGeorge2016} and  \cite[formula (18)]{DimitrovGeorge2016}, there exists a homeomorphism which we denote by $f_{\CE}$:
\begin{eqnarray}\label{homeo}
f_{\CE}:\Theta_{\CE}\longrightarrow \BR_{>0}^{n+1}\times \ykh{\bigcup_{\textbf{p}\in A_0}S^n(-1,+1)-\textbf{p}}.
\end{eqnarray}

\begin{remark}
Dimitrov--Katzarkov \cite[p. 833-834]{DimitrovGeorge2016}  defined the homeomorphism $f_{\CE}$ by the restriction. The exact definition of $f_{\CE}$ is not necessary for our proof. For the details of $f_{\CE}$ the reader is referred to the original paper \cite{DimitrovGeorge2016}. 
\end{remark}

For $n=2$ case,  Dimitrov--Katzarkov \cite{DimitrovGeorge2016} gave an exact description of $\Theta_{\CE}$. In the next section we give a complete description of $\Theta_{\CE}$ for $n=3$ case using different menthod.


\section{Main Theorem}\label{mainproof}
In this section we prove the main theorem. For $x\in\BR$, let us denote \emph{floor function} by 
\[
\xqz{x}=\max\hkh{n\in \BZ: n\leq x},
\]
\emph{ceiling function} by $\sqz{x}=\xqz{x}+1$ and  \emph{fractional part} of $x$ by $\hkh{x}=x-\xqz{x}$.
It is easy to check that for every $x$ and $y$, the following inequality holds
\begin{equation}\label{xqz}
\xqz{x}+\xqz{y}\leq \xqz{x+y}\leq \xqz{x}+\xqz{y}+1.
\end{equation}

For $(x_0,x_1,x_2,x_3)\in\BR^4, 1\leq i<j\leq 3$ and $\alpha\in\hkh{0,1}$, by the inequality \eqref{xqz}, we can define $\CR_{ij}^{\alpha}$ as the relation
\begin{equation}\label{relationR}
\xqz{x_0-x_j }-\xqz{x_0-x_i }=\xqz{x_i-x_j }+\alpha.
\end{equation}

By the inequality \eqref{xqz},  the following lemma holds.

\begin{lemma}\label{pair1}
For fixed $i, j$, $\CR_{ij}^0$ and $\CR_{ij}^1$ do not hold at the same time. 
\end{lemma}

Moreover, for triple relations  $\CR_{12}^{\alpha}$,  $\CR_{13}^{\beta}$ and $\CR_{23}^{\gamma}$,  the following lemma holds.
\begin{lemma}\label{triple1} 
Let $(x_0,x_1,x_2,x_3)\in\BR^4$,  $\alpha,\beta,\gamma\in\hkh{0,1}$. If relations $\CR_{12}^{\alpha}$,  $\CR_{13}^{\beta}$ and $\CR_{23}^{\gamma}$ hold at the same time. Then
\begin{equation}
\beta\leq\alpha+\gamma\leq \beta+1.\label{abcinequ}
\end{equation}
In particular, all possibly values of  triple $(\alpha,\beta,\gamma)$ are 
\[
(0,0,0),\quad (1,1,0),\quad (1,0,0),\quad (0,0,1), \quad (1,1,1),\quad (0,1,1).
\]
\end{lemma}

\begin{proof}
If relations $\CR_{12}^{\alpha}$,  $\CR_{13}^{\beta}$ and $\CR_{23}^{\gamma}$ hold at the same time, i.e.
\begin{align*}
&\xqz{x_0-x_2}-\xqz{x_0-x_1}=\xqz{x_1-x_2}+\alpha\\
&\xqz{x_0-x_3}-\xqz{x_0-x_1}=\xqz{x_1-x_3}+\beta\\
&\xqz{x_0-x_3}-\xqz{x_0-x_2}=\xqz{x_2-x_3}+\gamma
\end{align*}
then we have
\[
\xqz{x_1-x_3}+\beta=\xqz{x_0-x_3}-\xqz{x_0-x_1}=\xqz{x_2-x_3}+\xqz{x_1-x_2}+\alpha+\gamma
\]
By the inequality \eqref{xqz}, we have
\[
\xqz{x_2-x_3}+\xqz{x_1-x_2}\leq \xqz{x_1-x_3}\leq \xqz{x_2-x_3}+\xqz{x_1-x_2}+1
\]
Hence
\[
\beta\leq\alpha+\gamma\leq \beta+1.
\]
Then we can list all possibly values of $(\alpha,\beta,\gamma)$ as follows
\[
(0,0,0),\quad (1,1,0),\quad (1,0,0),\quad (0,0,1), \quad (1,1,1),\quad (0,1,1).
\]
\end{proof}

Then we have following corollary

\begin{coro}\label{triple2}
Relations $\CR_{12}^0$,  $\CR_{13}^1$ and $\CR_{23}^0$ do not hold at the some time. Relations  $\CR_{12}^1$,  $\CR_{13}^0$ and $\CR_{23}^1$ do not hold at the some time.
\end{coro}

Before the main theorem, we recall more properties of $\CR_{ij}^{\alpha},\alpha\in\hkh{0,1}$.

\begin{lemma}\label{triple3}
Let $(x_0,x_1,x_2,x_3)\in\BR^4$. If $\xqz{x_i-x_j}= x_i-x_j$ hold, then $\CR_{ij}^{0}$ hold.
\end{lemma}

\begin{proof}
If $\xqz{x_i-x_j}= x_i-x_j$ holds, then $\hkh{x_i}=\hkh{x_j}$, we also have  $\hkh{x_0-x_i}=\hkh{x_0-x_j}$. By the definition of $\hkh{x}$, we have 
\[
x_0-x_i-\xqz{x_0-x_i}=x_0-x_j-\xqz{x_0-x_j}
\]
then $\CR_{ij}^{0}$ hold.
\end{proof}

Hence we have
\begin{coro}\label{pijcor}
If $\CR_{ij}^0$ holds,
 then
\[
\ykh{x_i+\xqz{x_0-x_i}}-\ykh{x_j+\xqz{x_0-x_j}}=x_i-x_j-\xqz{x_i-x_j}\in[0,1).
\]
If $\CR_{ij}^1$ holds, then $\xqz{x_i-x_j}\neq x_i-x_j$, and
\[
\ykh{x_i+\xqz{x_0-x_i}}-\ykh{x_j+\xqz{x_0-x_j}}=x_i-x_j-\xqz{x_i-x_j}-1\in(-1,0).
\]
\end{coro}


Before the main theorem, we prove the following four lemmas which is useful in the proof of the main theorem.

\begin{lemma}\label{lemmacase1}
If $\xqz{ x_0-x_1}=k_{01}-1$ and $\xqz{ x_0-x_2}=\CK_{02}-1$, then for any $(x_0,x_1,x_2,x_3)\in\Theta_1$, there exists $(0,p_1,p_2,p_3)\in A_0$ and $z_i\in S^3(-1,+1)$ such that $x_0=z_0, x_i=z_i-p_i, i=1,2,3$.
\end{lemma}

\begin{proof}
If $\xqz{ x_0-x_1}=k_{01}-1$ and $\xqz{ x_0-x_2}=\CK_{02}-1$, then for any $(x_0,x_1,x_2,x_3)\in\Theta_1$ , we can set
\[
p_i=\xqz{ x_0-x_i}, z_i=x_i+p_i, i=1,2,3.
\]
Then we have
\begin{eqnarray*}
p_2-p_1=\CK_{02}-1-k_{01}+1\leq k_{12}-1, \\
p_3-p_1\leq \CK_{03}-1-k_{01}+1\leq \CK_{13}-1,\\
p_3-p_2 \leq \CK_{03}-1-\CK_{02}+1\leq k_{23}-1
\end{eqnarray*}
Hence $(0,p_1,p_2,p_3)\in A_0$. By the Corollary \ref{pijcor}, we have $z_i\in S^3(-1,+1)$.
\end{proof}

\begin{lemma}\label{lemmacase2}
Let 
\[
\Delta_{\uppercase\expandafter{\romannumeral2}}=\hkh{(x_0,x_1,x_2,x_3)\in\Theta_1\Bigg|
\begin{array}{cll} 
&\quad\CR_{23}^1,\\
&\xqz{x_0-x_1}=k_{01}-1,\xqz{x_0-x_2}< \CK_{02}-1, \\
&\xqz{x_2-x_3}=k_{23}-1,x_1-x_2=\xqz{x_1-x_2} 
\end{array} 
}.
\]
If $\xqz{ x_0-x_1}=k_{01}-1\text{ and } \xqz{ x_0-x_2}<\CK_{02}-1$, then for any $(x_0,x_1,x_2,x_3)\in\Theta_1\setminus\Delta_{\uppercase\expandafter{\romannumeral2}}$, there exists $(0,p_1,p_2,p_3)\in A_0$ and $z_i\in S^3(-1,+1)$ such that $x_0=z_0, x_i=z_i-p_i, i=1,2,3$.
\end{lemma}

\begin{proof}
If $\xqz{ x_0-x_1}=k_{01}-1\text{ and } \xqz{ x_0-x_2}<\CK_{02}-1$,  by the Lemma \ref{pair1}, we can consider following cases

\begin{enumerate}
\item If the relation $\CR_{23}^0$ holds, then 
we can set
\[
p_i=\xqz{ x_0-x_i}, z_i=x_i+p_i, i=1,2,3
\]
Then we have
\begin{eqnarray*}
p_2-p_1\leq\CK_{02}-1-k_{01}+1\leq k_{12}-1, \\
p_3-p_1\leq \CK_{03}-1-k_{01}+1\leq \CK_{13}-1,\\
p_3-p_2=\xqz{x_0-x_3}-\xqz{x_0-x_2}=\xqz{x_2-x_3}\leq k_{23}-1
\end{eqnarray*}
Hence $(0,p_1,p_2,p_3)\in A_0$. By the Corollary \ref{pijcor}, we have $z_i\in S^3(-1,+1)$.

\item If the relation $\CR_{23}^1$ holds,

\begin{enumerate}[(i)]

\item\label{case221} If $\xqz{ x_2-x_3}\leq k_{23}-2$, we can set 
\[
p_i=\xqz{x_0-x_i}, z_i=x_i+p_i, i=1,2,3
\]
Then we have
\begin{eqnarray*}
p_2-p_1\leq\CK_{02}-1-k_{01}+1\leq k_{12}-1, \\
p_3-p_1\leq \CK_{03}-1-k_{01}+1\leq \CK_{13}-1,\\
p_3-p_2=\xqz{x_0-x_3}-\xqz{x_0-x_2}=\xqz{x_2-x_3}+1\leq k_{23}-1
\end{eqnarray*}
Hence $(0,p_1,p_2,p_3)\in A_0$. By the Corollary \ref{pijcor}, we have $z_i\in S^3(-1,+1)$.

\item If the relation $\CR_{12}^0$ and the condition $\xqz{ x_1-x_2}\neq x_1-x_2$  hold, 
we set
\[
p_1=\xqz{ x_0-x_1},  \quad p_2=\sqz{ x_0-x_2 } ,\quad p_3=\xqz{ x_0-x_3}, z_i=x_i+p_i, i=1,2,3
\]
Then we have
\begin{eqnarray*}
p_2-p_1\leq\CK_{02}-1-k_{01}+1\leq k_{12}-1, \\
p_3-p_1\leq \CK_{03}-1-k_{01}+1\leq \CK_{13}-1,\\
p_3-p_2=\xqz{x_0-x_3}-\xqz{x_0-x_2}=\xqz{x_2-x_3}+1\leq k_{23}-1
\end{eqnarray*}
Hence $(0,p_1,p_2,p_3)\in A_0$. By the Corollary \ref{pijcor}, we have $z_i\in S^3(-1,+1)$.
\end{enumerate}
\end{enumerate}
If $\CR_{23}^1$ and $\CR_{12}^1$ hold, then only case \ref{case221} is possible. Then we can write $\Delta_{\uppercase\expandafter{\romannumeral2}}$ by
\[
\Delta_{\uppercase\expandafter{\romannumeral2}}=\hkh{(x_0,x_1,x_2,x_3)\in\Theta_1\Bigg|
\begin{array}{cll} 
&\quad\CR_{23}^1,\\
&\xqz{x_0-x_1}=k_{01}-1,\xqz{x_0-x_2}< \CK_{02}-1, \\
&\xqz{x_2-x_3}=k_{23}-1,x_1-x_2=\xqz{x_1-x_2} 
\end{array} 
}
\]
\end{proof}

\begin{lemma}\label{lemmacase3}
Let $\Delta_{\uppercase\expandafter{\romannumeral3}}=\Delta_{\uppercase\expandafter{\romannumeral3}}^1\cup\Delta_{\uppercase\expandafter{\romannumeral3}}^2\cup\Delta_{\uppercase\expandafter{\romannumeral3}}^3$
where 
\begin{eqnarray*}
&\Delta_{\uppercase\expandafter{\romannumeral3}}^1=\hkh{(x_0,x_1,x_2,x_3)\in\Theta_1\Bigg|
\begin{array}{clll} 
&\quad\CR_{12}^1,\CR_{13}^0,\\
&\xqz{x_0-x_1}<k_{01}-1,\xqz{x_0-x_2}= \CK_{02}-1, \\
&\xqz{x_1-x_2}=k_{12}-1,x_2-x_3=\xqz{x_2-x_3} 
\end{array} 
}\\
&\Delta_{\uppercase\expandafter{\romannumeral3}}^2=\hkh{(x_0,x_1,x_2,x_3)\in\Theta_1\Bigg|
\begin{array}{clll} 
&\quad\CR_{12}^1,\CR_{13}^0,\\
&\xqz{x_0-x_1}<k_{01}-1,\xqz{x_0-x_2}= \CK_{02}-1, \\
&\xqz{x_1-x_2}=k_{12}-1,\xqz{x_0-x_3} =\CK_{03}-1
\end{array} 
}
\\
&\Delta_{\uppercase\expandafter{\romannumeral3}}^3=\hkh{(x_0,x_1,x_2,x_3)\in\Theta_1\Bigg|
\begin{array}{clll} 
&\quad\CR_{12}^0,\CR_{13}^1,\xqz{x_1-x_3}=\CK_{13}-1,\\
&\xqz{x_0-x_1}<k_{01}-1,\xqz{x_0-x_2}= \CK_{02}-1
\end{array}
}.
\end{eqnarray*}
If $\xqz {x_0-x_1}< k_{01}-1\text{ and } \xqz{x_0-x_2}=\CK_{02}-1$, then for any $(x_0,x_1,x_2,x_3)\in\Theta_1\setminus\Delta_{\uppercase\expandafter{\romannumeral3}}$, there exists $(0,p_1,p_2,p_3)\in A_0$ and $z_i\in S^3(-1,+1)$ such that $x_0=z_0, x_i=z_i-p_i, i=1,2,3$.
\end{lemma}

\begin{proof}
If $\xqz {x_0-x_1}< k_{01}-1\text{ and } \xqz{x_0-x_2}=\CK_{02}-1$, as the similar analysis in Lemma \ref{lemmacase2}, by Lemma \ref{pair1}, we can consider the following cases:

\begin{enumerate}
\item If the relation $\CR_{13}^0$ holds, 

\begin{enumerate}[(i)]
\item  If the relation $\CR_{12}^0$ holds, then 
we can find 
\[
p_i=\xqz{ x_0-x_i}, z_i=x_i+p_i, i=1,2,3
\]
such that $(0,p_1,p_2,p_3)\in A_0$ and $z_i\in S^3(-1,+1)$.

\item If  the relation $\CR_{12}^1$ and $\xqz{x_1-x_2}\leq k_{12}-2$ hold, then 
we can find 
\[
p_i=\xqz{ x_0-x_i},z_i=x_i+p_i, i=1,2,3
\]
such that $(0,p_1,p_2,p_3)\in A_0$ and $z_i\in S^3(-1,+1)$.

\item  If relations $\CR_{12}^1$, $\CR_{23}^0$ , $\xqz{x_2-x_3}\neq x_2-x_3$ and $\sqz{x_0-x_3}\leq\CK_{03}-1$ hold, then 
we can find 
\[
p_1=\sqz{ x_0-x_1}, p_2=\xqz{ x_0-x_2},p_3=\sqz{ x_0-x_3},z_i=x_i+p_i, i=1,2,3
\]
such that $(0,p_1,p_2,p_3)\in A_0$ and $z_i\in S^3(-1,+1)$.
\end{enumerate}

\item If the relation $\CR_{13}^1$ holds,

\begin{enumerate}[(i)]
\item If $\CR_{12}^0$ and $x_1-x_3\neq \xqz{x_1-x_3}\leq \CK_{13}-2$ hold, we can find 
\[
p_i=\xqz{x_0-x_i}, z_i=x_i+p_i, i=1,2,3
\]
such that $(0,p_1,p_2,p_3)\in A_0$ and $z_i\in S^3(-1,+1)$.

\item If $\CR_{12}^1$ holds, we can find 
\[
p_1=\sqz{x_0-x_1}, p_2=\xqz{x_0-x_2}, p_3=\xqz{x_0-x_3}, \quad z_i=x_i+p_i, i+1,2,3
\]
such that $(0,p_1,p_2,p_3)\in A_0$ and $z_i\in S^3(-1,+1)$.
\end{enumerate}
\end{enumerate}

By Lemma \ref{triple3}, if $\xqz{x_1-x_3} = x_1-x_3$, then the relation $\CR_{13}^0$ holds. Hence we can write  $\Delta_{\uppercase\expandafter{\romannumeral3}}$ as 
\[
\Delta_{\uppercase\expandafter{\romannumeral3}}=\Delta_{\uppercase\expandafter{\romannumeral3}}^1\cup\Delta_{\uppercase\expandafter{\romannumeral3}}^2\cup\Delta_{\uppercase\expandafter{\romannumeral3}}^3
\]
where 
\begin{eqnarray*}
&\Delta_{\uppercase\expandafter{\romannumeral3}}^1=\hkh{(x_0,x_1,x_2,x_3)\in\Theta_1\Bigg|
\begin{array}{clll} 
&\quad\CR_{12}^1,\CR_{13}^0,\\
&\xqz{x_0-x_1}<k_{01}-1,\xqz{x_0-x_2}= \CK_{02}-1, \\
&\xqz{x_1-x_2}=k_{12}-1,x_2-x_3=\xqz{x_2-x_3} 
\end{array} 
}\\
&\Delta_{\uppercase\expandafter{\romannumeral3}}^2=\hkh{(x_0,x_1,x_2,x_3)\in\Theta_1\Bigg|
\begin{array}{clll} 
&\quad\CR_{12}^1,\CR_{13}^0,\\
&\xqz{x_0-x_1}<k_{01}-1,\xqz{x_0-x_2}= \CK_{02}-1, \\
&\xqz{x_1-x_2}=k_{12}-1,\xqz{x_0-x_3} =\CK_{03}-1
\end{array} 
}
\\
&\Delta_{\uppercase\expandafter{\romannumeral3}}^3=\hkh{(x_0,x_1,x_2,x_3)\in\Theta_1\Bigg|
\begin{array}{clll} 
&\quad\CR_{12}^0,\CR_{13}^1,\xqz{x_1-x_3}=\CK_{13}-1,\\
&\xqz{x_0-x_1}<k_{01}-1,\xqz{x_0-x_2}= \CK_{02}-1
\end{array} 
}
\end{eqnarray*}
\end{proof}

\begin{lemma}\label{lemmacase4}
Let $\Delta_{\uppercase\expandafter{\romannumeral4}}=\Delta_{\uppercase\expandafter{\romannumeral4}}^1\cup\Delta_{\uppercase\expandafter{\romannumeral4}}^2\cup\Delta_{\uppercase\expandafter{\romannumeral4}}^3$
where
\begin{eqnarray*}
&\Delta_{\uppercase\expandafter{\romannumeral4}}^1=\hkh{(x_0,x_1,x_2,x_3)\in\Theta_1\Bigg|
\begin{array}{cll} 
&\quad\CR_{12}^1,\CR_{13}^0,\CR_{23}^0,\\
&\xqz{x_0-x_1}<k_{01}-1,\xqz{x_0-x_2}< \CK_{02}-1, \\
&\xqz{x_1-x_2}=k_{12}-1,\xqz{x_2-x_3}=x_2-x_3, 
\end{array} 
}\\
&\Delta_{\uppercase\expandafter{\romannumeral4}}^2=\hkh{(x_0,x_1,x_2,x_3)\in\Theta_1\Bigg|
\begin{array}{cl} 
&\quad\CR_{12}^1,\CR_{13}^0,\CR_{23}^0,\\
&\xqz{x_0-x_1}<k_{01}-1,\xqz{x_0-x_2}< \CK_{02}-1, \\
&\xqz{x_1-x_2}=k_{12}-1,\xqz{x_0-x_3}=\CK_{03}-1, 
\end{array} 
}\\
&\Delta_{\uppercase\expandafter{\romannumeral4}}^3=\hkh{(x_0,x_1,x_2,x_3)\in\Theta_1\Bigg|
\begin{array}{cl} 
&\quad\CR_{12}^0,\CR_{13}^0,\CR_{23}^1,\\
&\xqz{x_0-x_1}<k_{01}-1,\xqz{x_0-x_2}<\CK_{02}-1, \\
&\xqz{x_1-x_2}=x_1-x_2,\xqz{x_2-x_3}=k_{23}-1, 
\end{array} 
}
\end{eqnarray*}

If $\xqz {x_0-x_1}< k_{01}-1\text{ and } \xqz{x_0-x_2}<\CK_{02}-1$, then for any $(x_0,x_1,x_2,x_3)\in\Theta_1\setminus\Delta_{\uppercase\expandafter{\romannumeral4}}$, there exists $(0,p_1,p_2,p_3)\in A_0$ and $z_i\in S^3(-1,+1)$ such that $x_0=z_0, x_i=z_i-p_i, i=1,2,3$.
\end{lemma}

\begin{proof}
If $\xqz {x_0-x_1}< k_{01}-1\text{ and } \xqz{x_0-x_2}<\CK_{02}-1$, as similar analysis in Lemma \ref{lemmacase2}, we can list all possible choices as follows

\begin{enumerate}[(A)]
\item\label{A120} If the relation $\CR_{12}^0$ holds,  then we have the following possible choices
\begin{enumerate}[(1)]
\item\label{A1201} $p_1=\xqz{x_0-x_1}$ and $p_2=\xqz{x_0-x_2}$,
\item\label{A1202} $p_1=\sqz{x_0-x_1}$ and $p_2=\sqz{x_0-x_2}$,
\item\label{A1203} If $x_1-x_2\neq \xqz{x_1-x_2} \text{ and } \xqz{x_1-x_2}\leq k_{12}-2$,
then we can choose
\[
 p_1=\xqz{x_0-x_1} \text{ and } p_2=\sqz{x_0-x_2}.
\]
\end{enumerate}
\item\label{A121} If the relation $\CR_{12}^1$ holds,  then we have the following possible choices

\begin{enumerate}[(1)]
\item\label{A1211} $p_1=\sqz{x_0-x_1}$ and $p_2=\xqz{x_0-x_2}$,
\item\label{A1212} If $\xqz{x_1-x_2}\leq k_{12}-2$ then $p_1=\xqz{x_0-x_1}$ and $p_2=\xqz{x_0-x_2}$, or
\item\label{A1213} If $\xqz{x_1-x_2}\leq k_{12}-2$ then $p_1=\sqz{x_0-x_1}$ and $p_2=\sqz{x_0-x_2}$,
\end{enumerate}
\item\label{A130} If the relation $\CR_{13}^0$ holds,  then we have the following possible choices

\begin{enumerate}[(1)]
\item\label{A1301} $p_1=\xqz{x_0-x_1}$ and $p_3=\xqz{x_0-x_3}$,
\item\label{A1302} If $\sqz{x_0-x_3}\leq k_{03}-1$, then we can choose $p_1=\sqz{x_0-x_1}$ and $p_3=\sqz{x_0-x_3}$,
\item\label{A1303} If $x_1-x_3\neq \xqz{x_1-x_3}, \xqz{x_1-x_3}\leq k_{13}-2,\text{ and } \sqz{x_0-x_3}\leq\CK_{03}-1$
then we can choose
\[
 p_1=\xqz{x_0-x_1} \text{ and } p_3=\sqz{x_0-x_3}.
\]
\end{enumerate}
\item\label{A131} If the relation $\CR_{13}^1$ holds,  then we have the following possible choices

\begin{enumerate}[(1)]
\item\label{A1311} $p_1=\sqz{x_0-x_1}$ and $p_3=\xqz{x_0-x_3}$,
\item\label{A1312} If $\xqz{x_1-x_3}\leq k_{13}-2$ then we can choose $p_1=\xqz{x_0-x_1}$ and $p_3=\xqz{x_0-x_3}$, or
\item\label{A1313} If $\xqz{x_1-x_3}\leq k_{13}-2$ and $\sqz{x_0-x_3}\leq\CK_{03}-1$ then we can choose $p_1=\sqz{x_0-x_1}$ and $p_3=\sqz{x_0-x_3}$,
\end{enumerate}
\item\label{A230} If the relation $\CR_{23}^0$ holds,  then we have the following possible choices
\begin{enumerate}[(1)]
\item\label{A2301} $p_2=\xqz{x_0-x_2}$ and $p_3=\xqz{x_0-x_3}$,
\item\label{A2302} If $\sqz{x_0-x_3}\leq\CK_{03}-1$, then we can choose $p_2=\sqz{x_0-x_2}$ and $p_3=\sqz{x_0-x_3}$,
\item\label{A2303} If $x_2-x_3\neq \xqz{x_2-x_3},  \xqz{x_2-x_3}\leq k_{23}-2, \text{ and } \sqz{x_0-x_3}\leq\CK_{03}-1$
then we can choose
\[
 p_2=\xqz{x_0-x_2} \text{ and } p_3=\sqz{x_0-x_3}.
\]
\end{enumerate}
\item\label{A231} If the relation $\CR_{23}^1$ holds,  then we have the following possible choices

\begin{enumerate}[(1)]
\item\label{A2311} $p_2=\sqz{x_0-x_2}$ and $p_3=\xqz{x_0-x_3}$,
\item\label{A2312} If $\xqz{x_2-x_3}\leq k_{23}-2$, then we can choose $p_2=\xqz{x_0-x_2}$ and $p_3=\xqz{x_0-x_3}$, or
\item\label{A2313} If $\xqz{x_2-x_3}\leq k_{23}-2$ and $\sqz{x_0-x_3}\leq\CK_{03}-1$, then we can choose $p_2=\sqz{x_0-x_2}$ and $p_3=\sqz{x_0-x_3}$.
\end{enumerate}
\end{enumerate} 

By  Lemma \ref{triple1}, Corollary \ref{triple2} , Corollary \ref{pijcor} and above analysis \eqref{A120}-\eqref{A121}-\eqref{A130}-\eqref{A131}-\eqref{A230}-\eqref{A231}, for this case, we have the following choices

\begin{enumerate}[(a)]
\item If $\CR_{12}^0$,  $\CR_{13}^0$ and $\CR_{23}^0$ hold, we can choose \eqref{A1201}-\eqref{A1301}-\eqref{A2301} 
we can find 
\[
p_i=\xqz{x_0-x_i}, z_i=x_i+p_i, i=1,2, 3,
\]
such that $(0,p_1,p_2,p_3)\in A_0$ and $z_i\in S^3(-1,+1)$.

\item If $\CR_{12}^0$,  $\CR_{13}^1$ and $\CR_{23}^1$ hold, we can choose \eqref{A1202}-\eqref{A1311}-\eqref{A2311}, we can find 
\[
p_1=\sqz{ x_0-x_1},\quad p_2=\sqz{ x_0-x_2} , \quad p_3=\xqz{x_0-x_3}, \quad z_i=x_i+p_i, i=1,2, 3
\]
such that $(0,p_1,p_2,p_3)\in A_0$ and $z_i\in S^3(-1,+1)$.

\item If $\CR_{12}^1$,  $\CR_{13}^1$ and $\CR_{23}^0$ hold, we can choose \eqref{A1211}-\eqref{A1311}-\eqref{A2301}, we can find 
\[
p_1=\sqz{x_0-x_1},  \quad p_2=\xqz{ x_0-x_2}, \quad p_3=\xqz{ x_0-x_3} \quad z_i=x_i+p_i, i=1, 2,3,
\]
such that $(0,p_1,p_2,p_3)\in A_0$ and $z_i\in S^3(-1,+1)$.

\item 
If $\CR_{12}^1$,  $\CR_{13}^0$ and $\CR_{23}^0$ hold,  

\begin{enumerate}[(i)]
\item if $\xqz{x_2-x_3}\leq k_{23}-2, x_2-x_3\neq \xqz{x_2-x_3} \text{ and } \xqz{x_0-x_3}\leq \CK_{03}-2,$
we can choose \eqref{A1211}-\eqref{A1302}-\eqref{A2303}, 
we can find 
\[
p_1=\sqz{x_0-x_1},  \quad p_2=\xqz{ x_0-x_2}, \quad p_3=\sqz{ x_0-x_3}, z_i=x_i+p_i, i=1,2,3
\]
such that $(0,p_1,p_2,p_3)\in A_0$ and $z_i\in S^3(-1,+1)$,

\item if $\xqz{x_1-x_2}\leq k_{12}-2$, we can choose \eqref{A1212}-\eqref{A1301}-\eqref{A2301} ,we can find
\[
p_i=\xqz{x_0-x_i}, z_i=x_i+p_i, i=1,2,3
\]
or moreover $\xqz{x_0-x_3}\leq \CK_{03}-2$, we can also choose \eqref{A1213}-\eqref{A1302}-\eqref{A2302}, we can find
\[
p_i=\sqz{x_0-x_i}, z_i=x_i+p_i, i=1,2,3
\]
such that $(0,p_1,p_2,p_3)\in A_0$ and $z_i\in S^3(-1,+1)$,
\end{enumerate}

\item 
If $\CR_{12}^0$,  $\CR_{13}^0$ and $\CR_{23}^1$ hold,  

\begin{enumerate}[(i)]
\item
if $\xqz{x_2-x_3}\leq k_{23}-2$, we can choose \eqref{A1201}-\eqref{A1301}-\eqref{A2312},
we can find
\[
p_i=\xqz{x_0-x_i}, z_i=x_i+p_i, i=1,2,3
\]
or moreover if $\sqz{x_0-x_3}\leq \CK_{03}-1$, we can also choose \eqref{A1202}-\eqref{A1302}-\eqref{A2313},
we can find
\[
p_i=\sqz{x_0-x_i}, z_i=x_i+p_i, i=1,2,3
\]
such that $(0,p_1,p_2,p_3)\in A_0$ and $z_i\in S^3(-1,+1)$,

\item
if $x_1-x_2\neq\xqz{x_1-x_2}$ , $\xqz{x_1-x_2}\leq k_{12}-2$, 
we can choose \eqref{A1203}-\eqref{A1301}-\eqref{A2311},
we can find
\[
p_1=\xqz{x_0-x_1}, p_2=\sqz{x_0-x_2}, p_3=\xqz{x_0-x_3}, z_i=x_i+p_i, i=1,2,3
\]
such that $(0,p_1,p_2,p_3)\in A_0$ and $z_i\in S^3(-1,+1)$,
\end{enumerate}

\item 
If $\CR_{12}^1$,  $\CR_{13}^1$ and $\CR_{23}^1$ hold,  
\begin{enumerate}[(i)]
\item
if $\xqz{x_2-x_3}\leq k_{23}-2$, we can choose \eqref{A1211}-\eqref{A1311}-\eqref{A2312}, we can find
\[
p_1=\sqz{x_0-x_1}, p_2=\xqz{x_0-x_2}, p_3=\xqz{x_0-x_3}, z_i=x_i+p_i, i=1,2,3
\]
such that $(0,p_1,p_2,p_3)\in A_0$ and $z_i\in S^3(-1,+1)$,

\item
if $\xqz{x_1-x_2}\leq k_{12}-2$, we can choose \eqref{A1212}-\eqref{A1311}-\eqref{A2311}, we can find
\[
p_1=\sqz{x_0-x_1},p_2=\sqz{x_0-x_2},p_3=\xqz{x_0-x_3} z_i=x_i+p_i, i=1,2,3
\]
such that $(0,p_1,p_2,p_3)\in A_0$ and $z_i\in S^3(-1,+1)$.
\end{enumerate}
\end{enumerate}
If $\CR_{12}^1$,  $\CR_{13}^0$ and $\CR_{23}^0$ hold, i.e.
\begin{align*}
&\xqz{x_0-x_2}-\xqz{x_0-x_1}=\xqz{x_1-x_2}+1\\
&\xqz{x_0-x_3}-\xqz{x_0-x_1}=\xqz{x_1-x_3}\\
&\xqz{x_0-x_3}-\xqz{x_0-x_2}=\xqz{x_2-x_3}
\end{align*}
then we have
\[
\xqz{x_1-x_3}=\xqz{x_1-x_2}+\xqz{x_2-x_3}+1
\]
If moreover, 
\[
\xqz{x_1-x_2}=k_{12}-1,\xqz{x_2-x_3}=k_{23}-1
\]
then we have
\[
\xqz{x_1-x_3}=\xqz{x_1-x_2}+\xqz{x_2-x_3}+1=k_{12}+k_{23}-1
\]
But
\[
\xqz{x_1-x_3}\leq \CK_{13}-1\leq k_{12}+k_{23}-2
\]
This is a contraction. Hence, if $\xqz{x_1-x_2}=k_{12}-1$ and $\xqz{x_2-x_3}=k_{23}-1$ hold then $\CR_{12}^1$,  $\CR_{13}^0$ and $\CR_{23}^0$ does not hold at the same time. And the same reason, if $\xqz{x_1-x_2}=k_{12}-1$ and $\xqz{x_2-x_3}=k_{23}-1$ hold then 
\begin{itemize}
\item
$\CR_{12}^1$,  $\CR_{13}^1$ and $\CR_{23}^1$ does not hold at the same time.
\item 
$\CR_{12}^0$,  $\CR_{13}^0$ and $\CR_{23}^1$ does not hold at the same time.
\end{itemize}

By above analysis, we can write $\Delta_{\uppercase\expandafter{\romannumeral4}}$ by 
\[
\Delta_{\uppercase\expandafter{\romannumeral4}}=\Delta_{\uppercase\expandafter{\romannumeral4}}^1\cup\Delta_{\uppercase\expandafter{\romannumeral4}}^2\cup\Delta_{\uppercase\expandafter{\romannumeral4}}^3
\]
where
\begin{eqnarray*}
&\Delta_{\uppercase\expandafter{\romannumeral4}}^1=\hkh{(x_0,x_1,x_2,x_3)\in\Theta_1\Bigg|
\begin{array}{cll} 
&\quad\CR_{12}^1,\CR_{13}^0,\CR_{23}^0,\\
&\xqz{x_0-x_1}<k_{01}-1,\xqz{x_0-x_2}< \CK_{02}-1, \\
&\xqz{x_1-x_2}=k_{12}-1,\xqz{x_2-x_3}=x_2-x_3, 
\end{array} 
}\\
&\Delta_{\uppercase\expandafter{\romannumeral4}}^2=\hkh{(x_0,x_1,x_2,x_3)\in\Theta_1\Bigg|
\begin{array}{cl} 
&\quad\CR_{12}^1,\CR_{13}^0,\CR_{23}^0,\\
&\xqz{x_0-x_1}<k_{01}-1,\xqz{x_0-x_2}< \CK_{02}-1, \\
&\xqz{x_1-x_2}=k_{12}-1,\xqz{x_0-x_3}=\CK_{03}-1, 
\end{array} 
}\\
&\Delta_{\uppercase\expandafter{\romannumeral4}}^3=\hkh{(x_0,x_1,x_2,x_3)\in\Theta_1\Bigg|
\begin{array}{cl} 
&\quad\CR_{12}^0,\CR_{13}^0,\CR_{23}^1,\\
&\xqz{x_0-x_1}<k_{01}-1,\xqz{x_0-x_2}<\CK_{02}-1, \\
&\xqz{x_1-x_2}=x_1-x_2,\xqz{x_2-x_3}=k_{23}-1, 
\end{array} 
}
\end{eqnarray*}

\end{proof}

Using inequality \eqref{xqz}, Lemma \ref{triple1},  Lemma \ref{lemmacase1}, Lemma \ref{lemmacase2}, Lemma \ref{lemmacase3}, Lemma \ref{lemmacase4}, Corollary \ref{triple2} and Corollary \ref{pijcor}, we will prove the  Theorem \ref{thmmain} of this paper:

\begin{proof}
We write 
\[
\Theta_2=
\bigcup_{(0,p_1,p_2,p_3)\in A_0}\hkh{(x_0,x_1,x_2,x_3)\Big|
\begin{array}{cll} 
-1+p_i&<x_0-x_i<&1+p_i,  i=1,2,3\\
-1+p_k-p_j&<x_j-x_k<&1+p_k-p_j ,1\leq j<k\leq 3
\end{array} 
}.
\]

By the definition of $A_0$,  the following inequalities hold
\begin{align}
&p_1\leq k_{01}-1, \label{pinequ1}\\
& p_2\leq \CK_{02}-1, \label{pinequ2}\\
& p_3\leq\CK_{03}-1,\label{pinequ3}\\
&p_2-p_1\leq k_{12}-1, \label{pinequ4}\\
&p_3-p_2\leq k_{23}-1,\label{pinequ5}\\
&p_3-p_1\leq \CK_{13}-1\label{pinequ6}
\end{align}

By the above inequalities \eqref{pinequ1}-\eqref{pinequ6}, we have 
\[
\Theta_2\subseteq \Theta_1
\]

By the homeomorphism \eqref{homeo}, $\Theta_\CE$ is homeomorphism to $\BR_{>0}^3\times \Theta_2$. To prove the theorem, we need to show that for every $(x_0,x_1,x_2,x_3)\in\Theta_1\setminus \Delta$, there exist $(0,p_1,p_2,p_3)\in A_0$ and $z_i\in S^3(-1,+1)$ such that $(x_0,x_1,x_2,x_3)=(z_0,z_1,z_2,z_3)-(0,p_1,p_2,p_3)$ and $\Theta_2\cap\Delta=\emptyset$. By the definition of $S^3(-1,+1)$, we need to choose $z_i$ such that $z_i\in (x_0-1,x_0+1)$, hence the possible choices of $p_i$ are $p_i=\xqz{x_0-x_i}$ or $p_i=\sqz{x_0-x_i}$ and $z_i=x_i+p_i$. We need to check the inequalities \eqref{pinequ1}-\eqref{pinequ6} for $p_i$ and $z_i\in S^3(-1,+1)$.

By relations $\CR_{ij}^0$, $\CR_{ij}^1$, the inequalities \eqref{pinequ1}-\eqref{pinequ6} and Corollary \ref{pijcor}, we need to consider the following four cases.

\begin{enumerate}[(I)]
\item \label{case1}$\xqz{ x_0-x_1}=k_{01}-1$ and $\xqz{ x_0-x_2}=\CK_{02}-1$
\item \label{case2}$\xqz{ x_0-x_1}=k_{01}-1$ and $\xqz{ x_0-x_2}<\CK_{02}-1$
\item \label{case3}$\xqz{ x_0-x_1}<k_{01}-1$ and $\xqz{ x_0-x_2}=\CK_{02}-1$
\item \label{case4} $\xqz{ x_0-x_1}<k_{01}-1$ and $\xqz{ x_0-x_2}<\CK_{02}-1$
\end{enumerate}

By the above analysis, we can describe the set $\Delta$ as follows.

\begin{itemize}

\item For Case \ref{case1}, by Lemma \ref{lemmacase1}, there are no elements in $\Delta$.

\item For Case \ref{case2}, by Lemma \ref{lemmacase2}, Then the set of elements of $\Delta$ in this case is $\Delta_{\uppercase\expandafter{\romannumeral2}}$

\item For Case \ref{case3}, by Lemma \ref{lemmacase3}, Then the set of elements of $\Delta$ in this case is $\Delta_{\uppercase\expandafter{\romannumeral3}}$


\item For Case \ref{case4},  by Lemma \ref{lemmacase4} , Then the set of elements of $\Delta$ in this case is $\Delta_{\uppercase\expandafter{\romannumeral4}}$
\end{itemize}

By the Corollary \ref{triple2}, we can rewrite $\Delta_{\uppercase\expandafter{\romannumeral4}}^1,\Delta_{\uppercase\expandafter{\romannumeral4}}^2,\Delta_{\uppercase\expandafter{\romannumeral4}}^3$ as

\begin{eqnarray*}
&\Delta_{\uppercase\expandafter{\romannumeral4}}^1=\hkh{(x_0,x_1,x_2,x_3)\in\Theta_1\Bigg|
\begin{array}{cll} 
&\quad\CR_{12}^1,\CR_{13}^0,\\
&\xqz{x_0-x_1}<k_{01}-1,\xqz{x_0-x_2}< \CK_{02}-1, \\
&\xqz{x_1-x_2}=k_{12}-1,\xqz{x_2-x_3}=x_2-x_3, 
\end{array} 
}\\
&\Delta_{\uppercase\expandafter{\romannumeral4}}^2=\hkh{(x_0,x_1,x_2,x_3)\in\Theta_1\Bigg|
\begin{array}{cl} 
&\quad\CR_{12}^1,\CR_{13}^0,\\
&\xqz{x_0-x_1}<k_{01}-1,\xqz{x_0-x_2}< \CK_{02}-1, \\
&\xqz{x_1-x_2}=k_{12}-1,\xqz{x_0-x_3}=\CK_{03}-1, 
\end{array} 
}\\
&\Delta_{\uppercase\expandafter{\romannumeral4}}^3=\hkh{(x_0,x_1,x_2,x_3)\in\Theta_1\Bigg|
\begin{array}{cl} 
&\quad\CR_{13}^0,\CR_{23}^1,\\
&\xqz{x_0-x_1}<k_{01}-1,\xqz{x_0-x_2}<\CK_{02}-1, \\
&\xqz{x_1-x_2}=x_1-x_2,\xqz{x_2-x_3}=k_{23}-1, 
\end{array} 
}
\end{eqnarray*}

By the Corollary \ref{triple2}, we can rewrite $\Delta_{\uppercase\expandafter{\romannumeral2}}$ as $\Delta_{\uppercase\expandafter{\romannumeral2}}=\Delta^1\cup\Delta_{\uppercase\expandafter{\romannumeral2}}^2$,
where  
\begin{eqnarray*}
&\Delta^1=\hkh{(x_0,x_1,x_2,x_3)\in\Theta_1\Bigg|
\begin{array}{cll} 
&\quad\CR_{23}^1,\quad\CR_{13}^1\\
&\xqz{x_0-x_1}=k_{01}-1,\xqz{x_0-x_2}< \CK_{02}-1, \\
&\xqz{x_2-x_3}=k_{23}-1,x_1-x_2=\xqz{x_1-x_2} 
\end{array} 
}\\
&\Delta_{\uppercase\expandafter{\romannumeral2}}^2=\hkh{(x_0,x_1,x_2,x_3)\in\Theta_1\Bigg|
\begin{array}{cll} 
&\quad\CR_{23}^1,\quad\CR_{13}^0\\
&\xqz{x_0-x_1}=k_{01}-1,\xqz{x_0-x_2}< \CK_{02}-1, \\
&\xqz{x_2-x_3}=k_{23}-1,x_1-x_2=\xqz{x_1-x_2} 
\end{array} 
}
\end{eqnarray*}

Then we hvae
\begin{eqnarray*}
&\Delta^2:=\Delta_{\uppercase\expandafter{\romannumeral4}}^3\cup\Delta_{\uppercase\expandafter{\romannumeral2}}^2=\hkh{(x_0,x_1,x_2,x_3)\in\Theta_1\Bigg|
\begin{array}{cll} 
&\CR_{23}^1,\quad\CR_{13}^0,\quad \xqz{x_0-x_2}< \CK_{02}-1, \\
&\xqz{x_2-x_3}=k_{23}-1,x_1-x_2=\xqz{x_1-x_2} 
\end{array} 
}\\
&\Delta^3:=\Delta_{\uppercase\expandafter{\romannumeral3}}^1\cup\Delta_{\uppercase\expandafter{\romannumeral4}}^1=\hkh{(x_0,x_1,x_2,x_3)\in\Theta_1\Bigg|
\begin{array}{clll} 
&\CR_{12}^1,\CR_{13}^0, \quad \xqz{x_0-x_1}<k_{01}-1, \\
&\xqz{x_1-x_2}=k_{12}-1,x_2-x_3=\xqz{x_2-x_3} 
\end{array} 
}\\
&\Delta^4:=\Delta_{\uppercase\expandafter{\romannumeral3}}^2\cup\Delta_{\uppercase\expandafter{\romannumeral4}}^2=\hkh{(x_0,x_1,x_2,x_3)\in\Theta_1\Bigg|
\begin{array}{clll} 
&\CR_{12}^1,\CR_{13}^0, \quad \xqz{x_0-x_1}<k_{01}-1, \\
&\xqz{x_1-x_2}=k_{12}-1,\xqz{x_0-x_3} =\CK_{03}-1
\end{array} 
}
\end{eqnarray*}

For the convenience we rewrite $\Delta_{\uppercase\expandafter{\romannumeral3}}^3$ by $\Delta^5$.

Hence we prove that
\[
\Delta=\Delta_{\uppercase\expandafter{\romannumeral2}}
\cup \Delta_{\uppercase\expandafter{\romannumeral3}}\cup\Delta_{\uppercase\expandafter{\romannumeral4}}=\bigcup_{i=1}^5\Delta^i.
\]

By the construction of $\Delta$, we have $\Theta_2\cap \Delta=\emptyset$.
We complete the  proof.
\end{proof}


\section{Examples}\label{example}

In this section we give an application of the Theorem \ref{thmmain}. Let $X$ be a quadric surface, which isomorphism to $\BP^1\times \BP^1$, and  $\DB(X)$ is the bounded derived category of coherent sheaves on $X$.  We use the notation
\[
\CO(a,b)=\pi_1^*\CO_{\BP^1}(a)\otimes\pi_2^*\CO_{\BP^1}(b)
\]
where $\pi_1,\pi_2: X\to\BP^1$ are the projections.

Then $\CE=\ykh{\CO,\CO(1,0),\CO(0,1),\CO(1,1)}$ is a complete strong exceptional collection on $\DB(X)$. For $\CE$, we have
\[
k_{ij}=0,\CK_{02}=\CK_{13}=-1, \text{ and } \CK_{03}=-2.
\]

Then we can rewrite the Theorem \ref{thmmain} as follwes:

\begin{prop}
Let $X$ be a quadric surface, $\CE=\ykh{\CO,\CO(1,0),\CO(0,1),\CO(1,1)}$ be a complete strong exceptional collection on $\DB(X)$. 

Then 
\[
\Theta_1=\hkh{
(y_0,y_1,y_2,y_3)\Bigg|\begin{array}{cl} 
&y_0<y_1<y_2<y_3, \quad y_1-y_3<-1\\
&y_0-y_3<-2,\quad y_0-y_2<-1
\end{array} 
}.
\]
and we can express $\Theta_{\CE}$ as followes:
\[
\BR_{>0}^4\times \ykh{\Theta_1\setminus\Delta}
\]
where $\Delta$ is the union of following five sets
\begin{eqnarray*}
&\Delta^1=\hkh{(x_0,x_1,x_2,x_3)\in\Theta_1\Bigg|
\begin{array}{cll} 
&\xqz{x_1-x_3}=\xqz{x_0-x_3},\\
&\xqz{x_0-x_1}=-1,\xqz{x_0-x_2}=\xqz{x_0-x_3}<-2, \\
&\xqz{x_2-x_3}=-1,x_1-x_2=\xqz{x_1-x_2} 
\end{array} 
}\\
&\Delta^2=\hkh{(x_0,x_1,x_2,x_3)\in\Theta_1\Bigg|
\begin{array}{cll} 
& \xqz{x_0-x_3}=\xqz{x_0-x_2}< -2, \\
&\xqz{x_0-x_3}-\xqz{x_0-x_1}=\xqz{x_1-x_3},\\
&\xqz{x_2-x_3}=-1,x_1-x_2=\xqz{x_1-x_2} 
\end{array} 
}\\
&\Delta^3=\hkh{(x_0,x_1,x_2,x_3)\in\Theta_1\Bigg|
\begin{array}{clll} 
&\xqz{x_0-x_1}=\xqz{x_0-x_2}<-1, \\
&\xqz{x_0-x_3}-\xqz{x_0-x_1}=\xqz{x_1-x_3},\\
&\xqz{x_1-x_2}=-1,x_2-x_3=\xqz{x_2-x_3} 
\end{array} 
}\\
&\Delta^4=\hkh{(x_0,x_1,x_2,x_3)\in\Theta_1\Bigg|
\begin{array}{clll} 
&\xqz{x_0-x_1}=\xqz{x_0-x_2}<-1, \\
&\xqz{x_0-x_1}+\xqz{x_1-x_3}+3=0,\\
&\xqz{x_1-x_2}=-1,\xqz{x_0-x_3}=-3
\end{array} 
}\\
&\Delta^5=\hkh{(x_0,x_1,x_2,x_3)\in\Theta_1\Bigg|
\begin{array}{clll} 
&\xqz{x_0-x_1}<-1,\\
&\xqz{x_0-x_2}= -2,\xqz{x_1-x_3}= -2\\
&\xqz{x_0-x_1}=-\xqz{x_1-x_2}-2=\xqz{x_0-x_3}+1
\end{array} 
}.
\end{eqnarray*}
\end{prop}

\begin{remark}
The author's original motivation was to give a detail description of stability conditions on $X$ and excepted to get a contractible connected component of the spaces of stability conditions on $\DB(X)$.  

Compared to the $\BP^2$ case \cite{lichunyi2017,macri2004}, $\BP^1\times \BP^1$ case is more difficult. As compute in the Theorem \ref{thmmain}, $\Delta$ is the union of five sets. We can not give a simple description of $\Delta$ now. We wish to prove that $\Theta_2$ is contractible.  Another problem is that the strong exceptional collection on $\BP^1\times \BP^1$ is not stong after the mutation. The space of  stability condition generated by exceptional collectioon on $\BP^1\times \BP^1$  is still not clear.
\end{remark}

\subsection*{Acknowledgements} I would like to thank to Professor Sen Hu and Dr. Teng Huang for their encouragement and valuable comments.

\renewcommand{\refname}{Reference}
\bibliographystyle{unsrt}

\end{document}